\documentclass{louloupart}
\usepackage{pinlabel}
\usepackage{graphicx}
\title[Isomorphism problem for even Artin groups]{On the isomorphism problem for even Artin groups}
\author[R Blasco-Garc\'ia]{Rub\'en Blasco-Garc\'ia}
\givenname{Rub\'en}
\surname{Blasco-Grac\'ia}
\address{Rub\'en Blasco-Garc\'ia, Departamento de Matem\'aticas, IUMA, Universidad de Zaragoza, C. Pedro Cerbuna 12, 50009 Zaragoza, Spain}
\email{rubenb@unizar.es}	
\author[L Paris]{Luis Paris}
\givenname{Luis}
\surname{Paris}
\address{Luis Paris, IMB, UMR 5584, CNRS, Univ. Bourgogne Franche-Comté, 21000 Dijon, France}
\email{lparis@u-bourgogne.fr}

\newtheorem{thm}{Theorem}[section]
\newtheorem{lem}[thm]{Lemma}
\newtheorem{prop}[thm]{Proposition}
\newtheorem{corl}[thm]{Corollary}

\theoremstyle{definition}

\newtheorem*{acknow}{Acknowledgments}

\numberwithin{equation}{section}

\makeatletter
\renewcommand{\thefigure}{\ifnum \c@section>\z@ \thesection.\fi
 \@arabic\c@figure}
\@addtoreset{figure}{section}
\makeatother

\begin{document}

\def\N{\mathbb N} \def\MM{\mathcal M} \def\AA{\mathcal A}
\def\Z{\mathbb Z} \def\Gr{{\mathcal Gr}} \def\TGr{{\mathcal{TG}r}}
\def\EE{\mathcal E} \def\mod{{\rm mod\,}} \def\UU{\mathcal U}
\def\lk{{\rm lk}} \def\st{{\rm st}} \def\lb{{\rm lb}}
\def\K{\mathbb K} \def\supp{{\rm supp}} \def\F{\mathbb F}
\def\rk{{\rm rk}} \def\VV{\mathcal V} \def\id{{\rm id}}


\begin{abstract}
An even Artin group is a group which has a presentation with relations of the form $(st)^n=(ts)^n$ with $n\ge 1$.
With a group $G$ we associate a Lie $\Z$-algebra $\TGr(G)$.
This is the usual Lie algebra defined from the lower central series, truncated at the third rank. 
For each even Artin group $G$ we determine a presentation for $\TGr(G)$. 
Then we prove a criterion to determine whether two Coxeter matrices are isomorphic.
Let $c,d\in\N$ such that $c\ge1$, $d\ge2$ and $\gcd(c,d)=1$.
We show that, if two even Artin groups $G$ and $G'$ having presentations with relations of the form $(st)^n=(ts)^n$ with $n\in\{c\}\cup\{d^k\mid k\ge1\}$ are such that $\TGr(G)\simeq\TGr(G')$, then $G$ and $G'$ have the same presentation up to permutation of the generators.
On the other hand, we show an example of two non-isomorphic even Artin groups $G$ and $G'$ such that $\TGr(G)\simeq\TGr(G')$.
\end{abstract}

\maketitle


\section{Introduction}\label{sec1}

Let $S$ be a finite set.
A \emph{Coxeter matrix} over $S$ is a square matrix $M = (m_{s,t})_{s,t \in S}$ indexed by the elements of $S$, with coefficients in $\N \cup \{ \infty \}$, satisfying $m_{s,s} = 1$ for all $s \in S$, and $m_{s,t} = m_{t,s} \ge 2$ for all $s,t \in S$, $s \neq t$.
For $s,t \in S$ and $m$ an integer greater or equal to $2$ we denote by $\Pi (s,t, m)$ the word $sts \cdots$ of length $m$.
In other words, $\Pi (s,t,m) = (st)^{\frac{m}{2}}$ if $m$ is even and $\Pi (s,t,m) = (st)^{\frac{m-1}{2}}s$ if $m$ is odd. 
Let $M = (m_{s,t})_{s,t \in S}$ be a Coxeter matrix. 
The \emph{Artin group} associated with $M$ is the group $A[M]$ defined by the presentation 
\[
A[M] = \langle S \mid \Pi (s,t, m_{s,t}) = \Pi (t,s, m_{s,t}) \text{ for } s,t \in S,\ s \neq t,\ m_{s,t} \neq \infty \rangle\,.
\]
The \emph{Coxeter group} $W[M]$ associated with $M$ is the quotient of $A[M]$ by the relations $s^2=1$, $s \in S$.
We say that $A[M]$ (or $M$) is of \emph{spherical type} if $W[M]$ is finite, and we say that $A[M]$ (or $M$) is \emph{right-angled} if $m_{s,t} \in \{2, \infty\}$ for all $s,t \in S$, $s \neq t$.
On the other hand, we say that $A[M]$ (or $M$) is \emph{even} if $m_{s,t}$ is even for all $s,t \in S$, $s \neq t$.

There are few results proven for all Artin groups and the theory consists essentially on the study of more or less extended families.
Basic questions such as the existence of torsion or the existence of a solution to the word problem are still open for all these groups.
The two most studied and best understood families of Artin groups are the one of Artin groups of spherical type and the one of right-angled Artin groups.
The family of even Artin groups has not been so much studied (see, however, Blasco-Garc\'ia--Cogolludo-Agust\'in \cite{BlaCog1}, Blasco-Garc\'ia--Juh\'asz--Paris \cite{BlJuPa1} and Blasco-Garc\'ia--Mart\'inez-P\'erez--Paris \cite{BlMaPa1}), but we are convinced by its interest, in particular because even Artin groups share several interesting properties with right-angled Artin groups.
For instance, any even Artin group retracts into any of its parabolic subgroups.

Let $\MM$ be a family of Coxeter matrices, and let $\AA[\MM]=\{A[M]\mid M\in\MM\}$ be the corresponding family of Artin groups. 
A solution to the \emph{isomorphism problem} in $\AA[\MM]$ is an algorithm which, given two Coxeter matrices $M,N\in\MM$, determines whether $A[M]$ and $A[N]$ are isomorphic or not. 
On the other hand, we say that the family $\AA[\MM]$ is \emph{rigid} if, given two Coxeter matrices $M,N\in\MM$, we have $A[M]\simeq A[N]$ if and only if $M\simeq N$.
Note that any rigid family of Artin groups has an obvious solution to the isomorphism problem.
Note also that these two definitions can be extended to Coxeter groups and, more generally, to any object defined from a Coxeter matrix.

We know that the family of right-angled Artin groups is rigid (see Droms \cite{Droms1}), and we know that the family of Artin groups of spherical type is rigid (see Paris \cite{Paris1}).
On the other hand, examples of non-isomorphic Coxeter matrices having isomorphic associated Artin groups can be found in Brady--McCammond--M\"uhlherr--Neumann \cite{BrMcMuNe1}.
Apart from these works, there are very few significant results on the isomorphism problem for Artin groups.

We know by Bahls \cite[Theorem 5.4]{Bahls1} that the family of even Coxeter groups is rigid.
Furthermore, we do not know any example of two Coxeter matrices $M$ and $N$ such that $A[M]$ is isomorphic to $A[N]$ but $W[M]$ is not isomorphic to $W[N]$.
So, for this reason we conjecture that the family of even Artin groups is rigid.
This is the question that motivated our study.
Note that Bahls' \cite{Bahls1} ideas cannot be extended to the study of even Artin groups because they strongly use the study of finite subgroups and we know by Charney \cite{Charn1} and Charney--Davis \cite{ChaDav1} that even Artin groups have finite cohomological dimension, hence they are torsion free.
Our approach here is inspired by the work of Droms \cite{Droms1} and Kim--Makar-Limanov--Neggers \cite{KiMaNeRo1} on right-angled Artin groups and graph algebras.

Let $G$ be a discrete group and let $\{\gamma_kG\}_{k=1}^\infty$ be its lower central series.
One can associate to $G$ a graded Lie $\Z$-algebra $\Gr(G) = \oplus_{k=1}^\infty \Gr_k(G)$, where $\Gr_k(G)=\frac{\gamma_k G}{\gamma_{k+1} G}$ for all $k$ (see Lazard \cite{Lazar1}, Magnus--Karrass--Solitar \cite{MaKaSo1}, or Bass--Lubotzky \cite{BasLub1}).
In the present work we use the Lie algebra $\Gr(G)$ truncated at the third rank, that is, $\TGr(G) = \Gr(G)/I$, where $I=\oplus_{k=3}^\infty \Gr_k(G)$.
Using a variant of Droms' \cite{Droms1} proof based on Kim--Makar-Limanov--Neggers \cite{KiMaNeRo1} of the rigidity of the family of right-angled Artin groups (see also Koberda \cite[Theorem 6.4]{Kober1}), we show that, if $M$ and $N$ are two right-angled Coxeter matrices and $\TGr(A[M]) \simeq \TGr (A[N])$, then $M \simeq N$ (see Corollary \ref{corl4_2}).
This obviously implies that, if $A[M]$ is isomorphic to $A[N]$, then $M\simeq N$.
So, this shows (again) that the family of right-angled Artin groups is rigid.
This is the point of view we intend to extend.

In this paper we show that the Lie algebras $\TGr(G)$ differentiate some even Artin groups but not all of them.
In Section \ref{sec2} we determine a presentation for $\TGr(A[M])$, where $M$ is any even Coxeter matrix (see Theorem \ref{thm2_1}).
For $c,d\in\N$ such that $c\ge 1$, $d \ge 2$ and $\gcd (c,d)=1$, we denote by $\EE(c,d)$ the set of even Coxeter matrices $M=(m_{s,t})_{s,t\in S}$ satisfying $m_{s,t} \in \{2c, \infty\} \cup \{2d^r \mid r \ge 1\}$ for all $s,t \in S$, $s \neq t$.
In Section \ref{sec5} we prove that the family $\AA [\EE(c,d)]$ is rigid (see Corollary \ref{corl5_13}), and in Section \ref{sec6} we show two even Coxeter matrices $M_0$ and $N_0$ such that $A[M_0] \not\simeq A[N_0]$ and $\TGr (A[M_0]) \simeq \TGr (A[N_0])$. 
The proof of Corollary \ref{corl5_13} uses a criterion proved in Section \ref{sec3} to determine whether two Coxeter matrices are isomorphic, as well as a study made in Section \ref{sec4} of some graded Lie algebras associated to right-angled Coxeter matrices.
These two sections may interest some readers regardless of the rest of the paper. 

\begin{acknow}
The authors would like to thank Conchita Mart\'inez P\'erez for her comments about the lower central series in even Artin groups which were the seed of this work.
The first author was partially supported by the Departamento de Industria e Innovaci\'on del Gobierno de Arag\'on and Fondo Social Europeo PhD grant, the Spanish Government MTM2015-67781-P (MINECO/FEDER) and MTM2016-76868-C2-2-P and Grupo Algebra y Geometr\'ia from Gobierno de Arag\'on. 
He was also supported by the grant CB 2/18 from Programa CAI-Ibercaja de Estancias de Investigaci\'on sponsored by Universidad de Zaragoza, Fundaci\'on Bancaria Ibercaja and Fundaci\'on CAI.
\end{acknow}


\section{Truncated Lie algebra associated to the lower central series of an even Artin group}\label{sec2}

Let $G$ be a group.
For $\alpha, \beta \in G$ we denote by $[\alpha, \beta] = \alpha^{-1} \beta^{-1} \alpha \beta$ the commutator of $\alpha$ and $\beta$.
If $H_1$ and $H_2$ are two subgroups of $G$, then we denote by $[H_1, H_2]$ the subgroup of $G$ generated by $\{ [\alpha, \beta] \mid \alpha \in H_1,\ \beta \in H_2\}$.
The \emph{lower central series} of $G$ is the sequence $\{ \gamma_k G \}_{k=1}^\infty$ of subgroups of $G$ recursively defined by $\gamma_1 G = G$ and $\gamma_k G = [G,\gamma_{k-1} G]$ for $k \ge 2$.
For $k \ge 1$ the quotient $\Gr_k(G) = \frac{\gamma_k G}{\gamma_{k+1} G}$ is an abelian group, so it is a $\Z$-module.
We set $\Gr(G) = \oplus_{k=1}^\infty \Gr_k (G)$.
For $\alpha \in \gamma_k G$, we denote by $\bar \alpha$ the element of $\Gr_k (G)$ represented by $\alpha$.
Then $\Gr (G)$ is endowed with a structure of graded Lie $\Z$-algebra as follows. 
If $\alpha \in \gamma_k G$ and $\beta \in \gamma_\ell G$, then the bracket $[\bar \alpha, \bar \beta]$ is the element of $\Gr_{k + \ell} (G)$ represented by the commutator $[\alpha, \beta]$.
We refer to Lazard \cite{Lazar1}, Magnus--Karrass--Solitar \cite{MaKaSo1}, and Bass--Lubotzky \cite{BasLub1} for detailed accounts on this Lie algebra.

We denote by $\TGr(G) = \Gr(G)/I$ the quotient of $\Gr (G)$ by the ideal $I= \oplus_{k=3}^\infty \Gr_k (G)$.
We have $\TGr(G) =_\Z \Gr_1 (G) \oplus \Gr_2 (G)$, and the Lie bracket on $\TGr(G)$ is defined as follows.
If $\alpha \in \gamma_k G$ and $\beta \in \gamma_\ell G$, $1 \le k,\ell \le 2$, then $[\bar \alpha, \bar \beta]$ is the element of $\Gr_{k+\ell} (G)$ represented by $[\alpha, \beta]$ if $k=\ell=1$, and $[\bar \alpha, \bar \beta]=0$ otherwise.
It is clear that, if $G$ is isomorphic to $G'$, then $\TGr(G)$ is isomorphic to $\TGr(G')$.

Now, take an even Coxeter matrix $M = (m_{s,t})_{s,t \in S}$ and fix a total order $<$ on $S$.
For $m \in \N \cup \{\infty\}$ we set $E_m(M)=\{ (s,t) \in S \times S \mid s < t \text{ and } m_{s,t} = m\}$.
We set
\begin{gather*}
L_2[M] = \left( \bigoplus_{2 \le n < \infty} \left( \bigoplus_{(s,t) \in E_{2n}(M)} (\Z/n\Z)\, v_{s,t} \right) \right) \oplus \left( \bigoplus_{(s,t) \in E_\infty(M)} \Z\, v_{s,t}\right)\,,\\ 
L_1[M] = \bigoplus_{s \in S} \Z u_s\,,\
L[M] = L_1 [M] \oplus L_2[M]\,.
\end{gather*} 
We define a Lie bracket $[\cdot,\cdot]$ on $L[M]$ as follows. 
Let $s,t \in S$, $s \neq t$.
Then $[u_s,u_t]=-[u_t,u_s] = v_{s,t}$ if $s<t$ and $m_{s,t} \neq 2$, and $[u_s,u_t] = 0$ if $m_{s,t}=2$.
On the other hand, we set $[\alpha,\beta] = 0$ for each $(\alpha,\beta) \in L[M] \times L_2 [M]$.
It is clear that $L[M]$ is a well-defined graded Lie $\Z$-algebra.

\begin{thm}\label{thm2_1}
Let $M$ be an even Coxeter matrix.
Then $\TGr (A[M])$ is isomorphic to $L[M]$.
\end{thm}

\begin{proof}
{\it Claim 1.}
Let $G$ be a group, let $\alpha, \beta \in G$, and let $n \in \N$, $n \ge 1$.
Then 
\[
(\beta \alpha)^{-n} (\alpha \beta)^n \equiv [\alpha,\beta]^n\ (\mod \gamma_3G)\,.
\]

{\it Proof of Claim 1.}
We argue by induction on $n$.
If $n=1$, then $(\beta \alpha)^{-1} (\alpha \beta) = [\alpha, \beta]$ and there is nothing to prove.
We suppose that $n \ge 2$ and that the inductive hypothesis holds.
Then
\[
(\beta \alpha)^{-n} (\alpha \beta)^n \equiv
\alpha^{-1} \beta^{-1} [\alpha, \beta]^{n-1} \alpha \beta \equiv
[\alpha, \beta]^{n-1} \alpha^{-1} \beta^{-1} \alpha \beta \equiv
[\alpha, \beta]^n\ (\mod \gamma_3G)\,.
\]
This concludes the proof of Claim 1.

{\it Claim 2.}
There is a surjective homomorphism $f : L[M] \to \TGr (A[M])$ which sends $u_s$ to $\bar s$ for all $s \in S$ and sends $v_{s,t}$ to $[\bar s, \bar t]$ for all $s,t \in S$ such that $s < t$ and $m_{s,t} \neq 2$.

{\it Proof of Claim 2.}
We know by Magnus--Karrass--Solitar \cite[Theorem 5.4]{MaKaSo1} that $\TGr_1 (A[M])$ is generated as a $\Z$-module by $\{ \bar s \mid s \in S\}$ and $\TGr_2 (A[M])$ is generated as a $\Z$-module by $\{ [\bar s, \bar t] \mid s,t \in S,\ s<t \}$.
Let $s,t \in S$, $s <t$.
If $m_{s,t}=2$, then $st = ts$, hence $[\bar s, \bar t]=0$.
If $m_{s,t}=2n$, where $2 \le n < \infty$, then $(st)^n = \Pi(s,t,m_{s,t}) = \Pi(t,s,m_{s,t}) = (ts)^n$, hence, by Claim 1, $n\,[\bar s, \bar t] = 0$.
So, there exists a well-defined homomorphism $f : L[M] \to \TGr (A[M])$ which sends $u_s$ to $\bar s$ for all $s \in S$ and sends $v_{s,t}$ to $[\bar s, \bar t]$ for all $s,t \in S$ such that $s < t$ and $m_{s,t} \neq 2$, and this homomorphism is surjective. 
This concludes the proof of Claim 2.

Let $\Z \langle X \rangle$ be the non-commutative associative free algebra over a set $X = \{x_s \mid s \in S\}$ in one-to-one correspondence with $S$.
Let $\hat I$ be the ideal of $\Z \langle X \rangle$ generated by $X$, and let $\hat R = \Z\langle X \rangle/\hat I^3$.
Let $J$ be the ideal of $\hat R$ generated by $\cup_{1 \le n < \infty} \{n\,(x_s x_t - x_t x_s) \mid (s,t) \in E_{2n}(M) \}$, and let $R = \hat R/J$.
Observe that $R$ is a graded associative $\Z$-algebra, $R = R_0 \oplus R_1 \oplus R_2$, where $R_0= \Z$, $R_1= \oplus_{s \in S} \Z x_s$, and 
\begin{gather*}
R_2 = \left( \bigoplus_{s \in S} \Z x_s^2 \right) \oplus \left( \bigoplus_{s<t} \Z x_t x_s \right) \oplus \left( \bigoplus_{2\le n <\infty} \left( \bigoplus_{(s,t) \in E_{2n} (M)} (\Z/n\Z)(x_sx_t - x_tx_s) \right) \right)\\
 \oplus \left( \bigoplus_{(s,t) \in E_\infty (M)} \Z (x_s x_t - x_t x_s)\right)\,.
\end{gather*}

Let $I = R_1 \oplus R_2$ be the ideal of $R$ generated by $X$.
Set $\UU(R) = \{1 + w \mid w \in I\}$.
Note that $\UU (R)$ is a subgroup of the group of units of $R$; the inverse of an element $1+w$ of $\UU(R)$ is $1 - w + w^2$.

{\it Claim 3.}
Let $u, v \in I$ and $n \in \N$, $n \ge 1$.
Then 
\[
\big((1+u)(1+v)\big)^n = 1+nu+nv + S_{n-1}(u^2 + v^2 +vu) + S_n uv\,,
\]
where $S_n = \sum_{k=1}^n k = \frac{n (n+1)}{2}$.

{\it Proof of Claim 3.}
We argue by induction on $n$. 
The case $n=1$ being trivial, we may assume that $n \ge 2$ and that the inductive hypothesis holds.
Then 
\begin{gather*}
\big((1+u)(1+v)\big)^n \\
 = \big( 1+(n-1)u+(n-1)v + S_{n-2}(u^2 + v^2 +vu) + S_{n-1} uv \big) (1 + u + v + uv)\\
= 1+(n-1)u+(n-1)v + S_{n-2}(u^2 + v^2 +vu) + S_{n-1} uv + u + (n-1)u^2 + (n-1) vu 
\\ + v + (n-1) uv + (n-1) v^2 + uv
\\ = 1+nu+nv + S_{n-1}(u^2 + v^2 +vu) + S_n uv\,.
\end{gather*}
This concludes the proof of Claim 3.

{\it Claim 4.}
There exists a homomorphism $\tilde g : A[M] \to \UU(R)$ which sends $s$ to $1+x_s$ for all $s \in S$.

{\it Proof of Claim 4.}
Let $s,t \in S$, $s \neq t$, such that $m_{s,t} \neq \infty$.
Set $m_{s,t} = 2n$ where $1 \le n < \infty$.
Then, by Claim 3,
\begin{gather*}
\big((1+x_s)(1+x_t)\big)^n - \big((1+x_t)(1+x_s)\big)^n\\
= \big( 1 + nx_s + nx_t + S_{n-1}(x_s^2 + x_t^2 +x_t x_s) + S_n x_s x_t \big)\\ \qquad - \big( 1 + nx_t + nx_s + S_{n-1}(x_t^2 + x_s^2 +x_s x_t) + S_n x_t x_s  \big)\\
= n(x_s x_t - x_t x_s) = 0\,,
\end{gather*}
hence 
\begin{gather*}
\Pi((1+x_s),(1+x_t),m_{s,t}) = 
\big((1+x_s)(1+x_t)\big)^n\\
= \big((1+x_t)(1+x_s)\big)^n
= \Pi((1+x_t),(1+x_s),m_{s,t})\,.
\end{gather*}
This concludes the proof of Claim 4.

We denote by $g : \TGr(A[M]) \to \TGr (\UU(R))$ the homomorphism induced by $\tilde g$.
For $k \ge 1$ we set $\UU_k(R) =\{1+w \mid w \in I^k\}$.
Note that $\UU_k(R)$ is a subgroup of $\UU (R)$ and $\UU_k(R)=\{1\}$ if $k \ge 3$.

{\it Claim 5.}
\begin{itemize}
\item[(1)]
Let $k \ge 1$.
Then $\UU_k(R)$ is a normal subgroup of $\UU(R)$.
\item[(2)]
Let $k, \ell \ge 1$.
Then $[\UU_k (R), \UU_\ell(R)] \subset \UU_{k+\ell}(R)$.
\end{itemize}

{\it Proof of Claim 5.}
Let $\alpha \in \UU(R)$ and $w \in I^k$.
Then $\alpha (1+w)\alpha^{-1} =1 + \alpha w \alpha^{-1} \in \UU_k(R)$.
So, $\UU_k (R)$ is a normal subgroup of $\UU (R)$.

Let $k, \ell \ge 1$, $u \in I^k$ and $v \in I^\ell$.
Then 
\begin{gather*}
[(1+u),(1+v)] =
(1+u)^{-1} (1+v)^{-1} (1+u) (1+v)\\
= (1 - u - v + u^2 + v^2 +uv) (1 + u + v + uv)\\
= 1 - u - v + u^2 + v^2 +uv + u - u^2 - vu + v - uv - v^2 + uv\\
= 1 +uv - vu \in \UU_{k+\ell}(R)\,.
\end{gather*}
This concludes the proof of Claim 5.

{\it End of the proof of Theorem \ref{thm2_1}.}
By Claim 5 we have $\gamma_k\UU (R) \subset\UU_k(R)$ for all $k \in \N$, and these inclusions induce for each $k$ a homomorphism $\Gr_k(\UU (R)) = \frac{\gamma_k \UU(R)}{\gamma_{k+1} \UU (R)} \to \frac{\UU_k(R)}{\UU_{k+1} (R)}$ of $\Z$-modules.
On the other hand, it is easily seen that $\frac{\UU_k(R)}{\UU_{k+1} (R)}$ is isomorphic to $R_k$ as a $\Z$-module, where $R_k = \{0\}$ if $k \ge 3$.
For $k=1,2$, by composing these homomorphisms we obtain a homomorphism $h_k : \TGr_k (\UU (R)) = \Gr_k (\UU (R)) \to R_k$ of $\Z$-modules.
A direct calculation shows that $(h_1 \circ g \circ f)(u_s) = x_s$ for all $s \in S$, hence $h_1 \circ g \circ f : L_1[M] \to R_1$ is injective, and therefore the restriction of $f$ to $L_1[M]$ is injective. 
Similarly, a direct calculation shows that, for each $s,t \in S$ such that $s < t$ and $m_{s,t} \neq 2$, we have $(h_2 \circ g \circ f)(v_{s,t}) = (x_s x_t - x_t x_s)$, hence $h_2 \circ g \circ f : L_2[M] \to R_2$ is injective, and therefore the restriction of $f$ to $L_2[M]$ is injective.
So, $f$ is an isomorphism.
\end{proof}


\section{Isomorphism between Coxeter matrices}\label{sec3}

Let $M$ and $N$ be two Coxeter matrices.
In the present section we prove a necessary and sufficient condition for $M$ and $N$ to be isomorphic (see Proposition \ref{prop3_4}).
This condition will be used later for even Coxeter matrices but it holds for any pair of Coxeter matrices. 
The case of right-angled Coxeter matrices is well-known and proved in Kim--Makar-Limanov--Neggers--Roush \cite{KiMaNeRo1}, with a proof that inspired ours.

Let $M = (m_{s,t})_{s,t \in S}$ be a Coxeter matrix. 
For $s \in S$ and $m \in \N$, $2 \le m < \infty$, the set $\lk_m(s) = \{ t \in S \setminus \{s\} \mid m_{s,t} \text{ divides } m\}$ is called the \emph{$m$-link} of $s$ and the set $\st_m(s) = \lk_m(s) \cup \{s\}$ is called the \emph{$m$-star} of $s$.
For $s,t \in S$ we set $s \prec_M t$ if $\lk_m(s) \subset \st_m(t)$ for all $m \in \N$, $2 \le m<\infty$.

\begin{lem}\label{lem3_1}
The relation $\prec_M$ is a quasi-order.
\end{lem}

\begin{proof}
The relation $\prec_M$ is clearly reflexive. 
So, we just need to show that $\prec_M$ is transitive.
Let $s,t,r \in S$ such that $s \prec_M t$ and $t \prec_M r$.
Let $m \in \N$ such that $2 \le m < \infty$. 
We have $\lk_m (s) \subset \st_m (t)$ and $\lk_m(t) \subset \st_m(r)$, and we have to show that $\lk_m(s) \subset \st_m(r)$.
If $s=t$ or $t = r$ or $s=r$, then there is nothing to prove, thus we may assume that $s \neq t$, $t \neq r$ and $s \neq r$.
If $m_{s,t}$ does not divide $m$, then $t \not\in \lk_m(s)$, hence $\lk_m(s) \subset \lk_m(t)$, and therefore $\lk_m(s) \subset \st_m(r)$.
So, we can suppose that $m_{s,t}$ divides $m$, that is, $s \in \lk_m(t)$ and $t \in \lk_m(s)$.
Since $\lk_m (t) \subset \st_m(r)$ and $s \neq r$, it follows that $s \in \lk_m(r)$, hence $m_{s,r}$ divides $m$ and $r \in \lk_m(s)$.
Since $\lk_m(s) \subset \st_m(t)$ and $r \neq t$, it follows that $r \in \lk_m(t)$, hence $m_{t,r}$ divides $m$ and $t \in \lk_m(r)$.
So, $\lk_m(s) \setminus\{t\} \subset \lk_m(t) \subset \st_m(r)$ and $t \in \lk_m(r) \subset \st_m(r)$, hence $\lk_m (s) \subset \st_m(r)$.
\end{proof}

Define the equivalence relation $\equiv_M$ on $S$ by  
\[
s \equiv_M t \ \Longleftrightarrow\ s \prec_M t \text{ and } t \prec_M s\,.
\]

\begin{lem}\label{lem3_2}
Let $C$ be an equivalence class of $\equiv_M$ non reduced to a single element. 
Then there exists $m \in \N \cup\{\infty\}$, $m \ge 2$, such that $m_{s,t} = m$ for all $s,t \in C$, $s \neq t$.
\end{lem}

\begin{proof}
Let $s, t \in C$, $s \neq t$, and let $m = m_{s,t}$.
Let $r \in C\setminus\{s,t\}$.
Suppose first that $m \neq \infty$.
We have $t \in \lk_m(s)$ and $\lk_m(s) \subset \st_m(r)$, hence $t \in \st_m(r)$.
Since $t\neq r$, it follows that $t \in \lk_m(r)$, thus $m_{t,r}$ divides $m=m_{s,t}$.
In particular, $m_{t,r} \neq \infty$.
Using the same argument we show that $m=m_{s,t}$ divides $m_{t,r}$, hence $m_{t,r} = m_{s,t}=m$.
We show that $m_{s,r}=m$ in the same way.
Now suppose that $m=\infty$.
If we had $m_{s,r} \neq \infty$, then, by the above, we would have $m_{s,t} = m_{s,r} \neq \infty$, a contradiction. 
Thus $m_{s,r} =\infty$.
Similarly, $m_{t,r} = \infty$.
\end{proof}

\begin{lem}\label{lem3_3}
Let $C,D$ be two distinct equivalence classes of $\equiv_M$.
Then there exists $m \in \N \cup \{ \infty\}$, $m \ge 2$, such that $m_{s,t} = m$ for all $s \in C$ and $t \in D$.
\end{lem}

\begin{proof}
We choose $s \in C$ and $t \in D$ and we set $m=m_{s,t}$.
We take $t' \in D$, $t' \neq t$, and we turn to show that $m_{s,t'} = m$.
Suppose first that $m \neq \infty$.
Since $C$ and $D$ are different, we have $s \neq t$ and $s \neq t'$.
Moreover, $s \in \lk_m(t)$ and $\lk_m(t) \subset \st_m(t')$, hence $s \in \lk_m(t')$, and therefore $m_{s,t'}$ divides $m=m_{s,t}$.
In particular $m_{s,t'} \neq \infty$.
We show that $m_{s,t}$ divides $m_{s,t'}$ with the same argument, hence $m_{s,t'} = m_{s,t}=m$.  
Now, suppose that $m = \infty$.
If we had $m_{s,t'} \neq \infty$, then, by the above, we would have $m_{s,t} = m_{s,t'} \neq \infty$, a contradiction.
So, $m_{s,t'} = \infty$.
\end{proof}

Thanks to Lemma \ref{lem3_3} we can define from $M$ a Coxeter matrix $M^r=(m_{C,D}^r)_{C,D \in S^r}$ as follows. 
\begin{itemize}
\item[(a)]
$S^r$ is the set of equivalence classes of $\equiv_M$.
\item[(b)]
Let $C,D \in S^r$.
Then choose $s \in C$ and $t\in D$, and set $m^r_{C,D} = m_{s,t}$.
\end{itemize}
Furthermore, thanks to Lemma \ref{lem3_2} we can define a map $\lb : S^r \to \N \cup \{ \infty\}$ as follows.
Let $C \in S^r$.
If $|C|=1$ we set $\lb(C)=0$.
If $|C| \ge 2$ we set $\lb(C)=m_{s,t}$, where $s,t$ are two distinct elements of $C$.

\begin{prop}\label{prop3_4}
Let $M=(m_{s,t})_{s,t \in S}$ and $N=(n_{x,y})_{x,y \in T}$ be two Coxeter matrices. 
If there exists an isomorphism $f: M^r \to N^r$ such that $|C| = |f(C)|$ and $\lb (C) = \lb (f(C))$ for all $C \in S^r$, then $M$ and $N$ are isomorphic.
\end{prop}

\begin{proof}
For each $C \in S^r$, since $|C| = |f(C)|$, we can choose a bijection $\hat f_C : C \to f(C)$.
Then we define $\hat f  : S \to T$ as follows. 
Let $s \in S$.
Let $C \in S^r$ such that $s \in C$.
Then we set $\hat f (s) = \hat f_C(s)$.
The map $\hat f: S \to T$ is a bijection by construction.
Moreover, it is easily seen that $m_{s,t} = n_{\hat f(s), \hat f(t)}$ for all $s,t \in S$, hence $\hat f$ is an isomorphism from $M$ to $N$.
\end{proof}


\section{Lie algebras associated to right-angled Artin groups}\label{sec4}

Let $M = (m_{s,t})_{s,t \in S}$ be a right-angled Coxeter matrix.
We assume $S$ to be endowed with a total order $\le$.
Let $\K$ be a field.
Recall that, for $m \in \{2, \infty\}$, we set $E_m(M)=\{ (s,t) \in S \times S \mid s < t \text{ and } m_{s,t} = m\}$.
We set 
\[
L_{\K,1}[M] = \bigoplus_{s \in S} \K u_s\,,\
L_{\K,2}[M] = \bigoplus_{(s,t) \in E_\infty(M)} \K\, v_{s,t}\,,\ 
L_\K[M] = L_{\K,1} [M] \oplus L_{\K,2}[M]\,.
\] 
We define a Lie bracket $[\cdot,\cdot]$ on $L_\K[M]$ as follows.
Let $s,t \in S$, $s \neq t$.
Then $[u_s,u_t]=-[u_t,u_s] = v_{s,t}$ if $s<t$ and $m_{s,t} =\infty$, and $[u_s,u_t] = 0$ if $m_{s,t}=2$.
On the other hand, we set $[\alpha,\beta] = 0$ for all $(\alpha, \beta) \in L_\K[M] \times L_{\K,2} [M]$.
It is clear that $L_\K[M]$ is a well-defined graded Lie $\K$-algebra.
In fact, $L_\K [M] = \K \otimes L[M]$.
The goal of the present section is to prove the following theorem.

\begin{thm}\label{thm4_1}
Let $M$ and $N$ be two right-angled Coxeter matrices and let $\K$ be a field.
If $L_\K[M]$ and $L_\K [N]$ are isomorphic, then $M$ and $N$ are isomorphic.
\end{thm}

\begin{corl}\label{corl4_2}
Let $M$ and $N$ be two right-angled Coxeter matrices.
If $L[M]$ and $L[N]$ are isomorphic, then $M$ and $N$ are isomorphic.
\end{corl}

\begin{proof}
Let $M$ and $N$ be two right-angled Coxeter matrices such that $L[M]$ and $L[N]$ are isomorphic. 
Then $\Q \otimes L[M] = L_\Q[M]$ is isomorphic to $\Q \otimes L[N] = L_\Q [N]$, hence, by Theorem \ref{thm4_1}, $M$ and $N$ are isomorphic. 
\end{proof}

\begin{corl}[Droms \cite{Droms1}]\label{corl4_3}
Let $\EE_{ra}$ be the set of right-angled Coxeter matrices. 
Then $\AA [\EE_{ra}]$ is rigid.
\end{corl}

\begin{proof}
Let $M$ and $N$ be two right-angled Coxeter matrices such that $A[M]$ and $A[N]$ are isomorphic. 
Then $\TGr (A[M])$ and $\TGr(A[N])$ are isomorphic, hence, by Theorem \ref{thm2_1}, $L[M]$ and $L[N]$ are isomorphic, and therefore, by Corollary \ref{corl4_2}, $M$ and $N$ are isomorphic. 
\end{proof}

The remaining of the section is devoted to the proof of Theorem \ref{thm4_1}.

Let $M=(m_{s,t})_{s,t \in S}$ be a right-angled Coxeter matrix. 
We associate to $M$ a graph $\Gamma_\infty[M]$ as follows.
\begin{itemize}
\item[(a)]
$S$ is the set of vertices of $\Gamma_\infty [M]$.
\item[(b)]
Two vertices $s,t \in S$, $s \neq t$, are connected by an edge if $m_{s,t} = \infty$.
\end{itemize}
For $X \subset S$ we denote by $\Gamma_{\infty,X} [M]$ the full subgraph of $\Gamma_\infty [M]$ spanned by $X$.
On the other hand, for $X \subset S$ we set $\lk_2(X) = \cap_{s \in X} \lk_2(s)$.

Let $\K$ be a field.
The \emph{support} of an element $\alpha = \sum_{s \in S} a_s u_s \in L_{\K,1} [M]$ is defined to be $\supp (\alpha) = \{ s \in S \mid a_s \neq 0 \}$.
On the other hand, the \emph{centralizer} of an element $\alpha \in L_{\K,1}[M]$ is $Z(\alpha) = \{\beta \in L_{\K,1} [M] \mid [\alpha, \beta] = 0\}$.

\begin{lem}\label{lem4_4}
Let $\K$ be a field, let $M=(m_{s,t})_{s,t \in S}$ be a right-angled Coxeter matrix, and let $\alpha = \sum_{s \in S} a_s u_s \in L_{\K,1} [M]$.
Let $X = \supp(\alpha)$, and let $\Omega_1, \dots, \Omega_\ell$ be the connected components of $\Gamma_{\infty,X} [M]$.
For each $i \in \{1, \dots, \ell\}$ we denote by $X_i$ the set of vertices of $\Omega_i$ and we set $\alpha_i = \sum_{s \in X_i} a_s u_s$.
Then 
\[
Z(\alpha) = \left( \bigoplus_{i=1}^\ell \K \alpha_i \right) \oplus \left(\bigoplus_{s \in \lk_2(X)} \K u_s \right)\,.
\]
\end{lem}

\begin{proof}
Set 
\[
V = \left( \bigoplus_{i=1}^\ell \K \alpha_i \right) \oplus \left(\bigoplus_{s \in \lk_2(X)} \K u_s \right)\,.
\]
The inclusion $V \subset Z(\alpha)$ is easily checked, so we just need to show that $Z(\alpha) \subset V$.
Let $\beta = \sum_{s \in S} b_s u_s \in Z(\alpha)$.
We have 
\begin{equation}\label{eq4_1}
0 = [\alpha, \beta] = \sum_{(s,t) \in E_\infty (M)} (a_s b_t - a_t b_s) v_{s,t}\,.
\end{equation}
Let $s \in S$ such that $s \not\in X$ and $s \not\in \lk_2(X)$.
There exists $t \in X$ such that $m_{s,t} = \infty$.
Suppose that $s < t$, that is, $(s,t) \in E_\infty(M)$.
By Equation (\ref{eq4_1}), $a_tb_s = 0$ since $a_s=0$, hence $b_s=0$ since $a_t \neq 0$.
Similarly, if $t<s$, then $b_s=0$.
Let $i \in \{1, \dots, \ell\}$.
If $s,t \in X_i$ are such that $(s,t) \in E_\infty(M)$, then, by Equation (\ref{eq4_1}), $a_s b_t = a_t b_s$, hence $\frac{b_s}{a_s} = \frac{b_t}{a_t}$.
Since $\Omega_i$ is connected, it follows that $\frac{b_s}{a_s} = \frac{b_t}{a_t}$ for all $s,t \in X_i$.
We choose $s \in X_i$ and we set $c_i = \frac{b_s}{a_s}$.
Then $\sum_{t \in X_i} b_t u_t = c_i \alpha_i$.
So, $\beta = ( \sum_{i=1}^\ell c_i \alpha_i ) + (\sum_{s \in \lk_2(X)} b_s u_s) \in V$.
\end{proof}

From now on we assume given a field $\K$, two right-angled Coxeter matrices $M = (m_{s,t})_{s,t \in S}$ and  $N=(n_{x,y})_{x,y \in T}$, and an isomorphism $f : L_\K[M] \to L_\K[N]$.
Our aim is to prove that $M$ and $N$ are isomorphic.

\begin{lem}\label{lem4_5}
Let $s \in S$ and $x,y \in T$ such that $n_{x,y}=2$ and $s \in \supp(f^{-1}(u_x))$.
Then $[u_y,f(u_s)]=0$.
\end{lem}

\begin{proof}
We set $f^{-1}(u_x) = \sum_{t \in S} a_t u_t$.
Let $X = \supp(f^{-1}(u_x))$, and let $\Omega_1, \dots, \Omega_\ell$ be the connected components of $\Gamma_{\infty,X} [M]$.
For each $i \in \{1, \dots, \ell\}$ we denote by $X_i$ the set of vertices of $\Omega_i$, and we set $\alpha_i = \sum_{t \in X_i} a_t u_t$.
We have $[u_x,u_y]=0$, hence $[f^{-1}(u_x), f^{-1}(u_y)]=0$, and therefore, by Lemma \ref{lem4_4}, $f^{-1}(u_y)$ is written in the form $f^{-1}(u_y) = \sum_{i=1}^\ell c_i \alpha_i + \gamma$, where $c_1, \dots, c_\ell \in \K$ and $\supp(\gamma) \subset \lk_2(X)$.
Since $s \in \supp(f^{-1}(u_x)) = X$, there exits $j \in \{1, \dots, \ell\}$ such that $s \in X_j$.
Then $[f^{-1}(u_y) - c_j\,f^{-1}(u_x),u_s]=0$, hence $[u_y-c_ju_x,f(u_s)]=0$.
If $c_j =0$, then the above equality becomes $[u_y,f(u_s)]=0$ and there is nothing else to prove.
So, we may assume that $c_j \neq 0$.
Since $n_{x,y}=2$, by applying again Lemma \ref{lem4_4} we get $f(u_s) = b_1u_y + b_2c_ju_x + \gamma'$, where $b_1,b_2 \in \K$ and $\supp(\gamma') \subset \lk_2(\{x,y\})$.
Then $\supp(f(u_s))\subset\st_2(u_y)$, hence $[u_y,f(u_s)]=0$ in this case as well. 
\end{proof}

For $s \in S$ and $x \in T$ we set $s \leftrightarrow_f x$ if $s \in \supp (f^{-1}(u_x))$ and $x \in \supp (f(u_s))$.

\begin{lem}\label{lem4_6}
Let $s \in S$.
Then there exists $x \in T$ such that $s \leftrightarrow_f x$.
\end{lem}

\begin{proof}
For $s \in S$ we set $f(u_s) = \sum_{x \in T} a_{s,x} u_x$ and for $x \in T$ we set $f^{-1}(u_x) = \sum_{s \in S} b_{x,s} u_s$.
Let $s \in S$.
Then  
\[
u_s = (f^{-1} \circ f)(u_s) = \sum_{x \in T} \sum_{t \in S} a_{s,x}b_{x,t}u_t\,,
\]
hence there exists $x \in T$ such that $a_{s,x} \neq 0$ and $b_{x,s} \neq 0$.
Then $s \in \supp (f^{-1} (u_x))$ and $x\in \supp (f(u_s))$, that is, $s \leftrightarrow_f x$.
\end{proof}

\begin{lem}\label{lem4_7}
Let $s,t \in S$, $s \neq t$, and $x,y \in T$, $x \neq y$, such that $s \leftrightarrow_f x$ and $t \leftrightarrow_f y$.
We have $m_{s,t}=2$ if and only if $n_{x,y}=2$.
\end{lem}

\begin{proof}
Suppose that $n_{x,y}=2$.
Since $s \leftrightarrow_f x$, we have $s \in \supp(f^{-1}(u_x))$.
Then, by Lemma \ref{lem4_5}, $[u_y, f(u_s)]=0$, hence $[f^{-1}(u_y), u_s]=0$.
By Lemma \ref{lem4_4} this implies that $\supp(f^{-1} (u_y)) \subset \st_2(s)$.
Since $t \in \supp(f^{-1}(u_y))$ and $t \neq s$ it follows that $t \in \lk_2(s)$, that is, $m_{s,t}=2$.
We show in the same way that, if $m_{s,t}=2$, then $n_{x,y}=2$.
\end{proof}

\begin{lem}\label{lem4_8}
Let $s,t \in S$ and $x,y \in T$ such that $s \leftrightarrow_f x$ and $t \leftrightarrow_f y$.
We have $s \prec_M t$ if and only if $x \prec_N y$.
\end{lem}

\begin{proof}
We suppose that $s \prec_M t$ and turn to prove that $x \prec_N y$.
If $x=y$, then there is nothing to prove. 
So, we can assume that $x \neq y$.
Let $z \in \lk_2 (x)$.
We have $n_{x,z}=2$ and $s \in \supp(f^{-1}(u_x))$ hence, by Lemma \ref{lem4_5}, $[u_z, f(u_s)]=0$, hence $[f^{-1}(u_z),u_s] = 0$.
By Lemma \ref{lem4_4} it follows that $f^{-1}(u_z)$ is of the form $f^{-1}(u_z) = au_s + \gamma$ where $a \in \K$ and $\supp(\gamma) \subset \lk_2(s)$.
Since $s \prec_M t$, we have $\lk_2(s) \subset \st_2(t)$, hence $[\gamma,u_t]=0$, and therefore $[u_z,f(u_t)] = [f(au_s+\gamma), f(u_t)] = f([au_s+\gamma,u_t]) = a\,f([u_s,u_t])$.

Suppose that $s \neq t$, $m_{s,t}=\infty$ and $a \neq 0$.
We have $t \not\in\supp (\gamma)$ since $\supp(\gamma) \subset \lk_2(s)$.
Let $r \in \supp(\gamma)$.
We have $r \in \lk_2(s) \subset \st_2(t)$ and $r \neq t$, hence $n_{r,t}=2$.
Moreover, $y \in \supp(f(u_t))$, hence, by Lemma \ref{lem4_5}, $[u_r,f^{-1}(u_y)]=0$.
So, $[\gamma, f^{-1}(u_y)]=0$, hence $[f(\gamma),u_y]=0$.
Since $x \in \supp(f(u_s))$ and $z \neq x$, the equality $u_z = a\, f(u_s) + f(\gamma)$ implies that $x \in \supp(f(\gamma))$.
By Lemma \ref{lem4_4} it follows that $x \in \lk_2(y)$, that is, $n_{x,y}=2$.
By Lemma \ref{lem4_7} we deduce that $m_{s,t}=2$, a contradiction.

So, either $s=t$, or $m_{s,t}=2$, or $a=0$.
In all theses cases we have $[u_z,f(u_t)] = a\,f([u_s,u_t]) = 0$.
Since $y \in \supp(f(u_t))$, it follows by Lemma \ref{lem4_4} that $z \in \st_2(y)$.
So, $\lk_2(x) \subset \st_2(y)$, that is, $x \prec_N y$.
We show in the same way that, if $x \prec_N y$, then $s \prec_M t$.
\end{proof}

\begin{corl}\label{corl4_9}
Let $s,t \in S$ and $x,y \in T$ such that $s \leftrightarrow_f x$ and $t \leftrightarrow_f y$.
We have $s \equiv_M t$ if and only if $x \equiv_N y$.
\end{corl}

Thanks to Corollary \ref{corl4_9} we can define a map $f^r : S^r \to T^r$ as follows. 
Let $C \in S^r$.
We choose $s \in C$ and $x \in T$ such that $s \leftrightarrow_f x$ and we set $f^r(C) = X$, where $X$ is the equivalence class of $x$ for the relation $\equiv_N$.
It is clear that $f^r$ is invertible and $(f^r)^{-1} = (f^{-1})^r$.
Moreover, by Lemma \ref{lem4_7}, $m_{C,D}^r = n_{f^r(C),f^r(D)}^r$ for all $C,D \in S^r$, hence $f^r$ is an isomorphism from $M^r$ to $N^r$.

\begin{lem}\label{lem4_10}
Let $s \in S$ and $x,y \in T$ such that $s \leftrightarrow_f x$ and $y \in \supp (f(u_s))$.
Then $x \prec_N y$.
\end{lem}

\begin{proof}
Let $z \in \lk_2(x)$.
We have $n_{x,z}=2$ and $s \in \supp(f^{-1}(u_x))$, hence, by Lemma \ref{lem4_5}, $[u_z, f(u_s)]=0$.
Since $y \in \supp(f(u_s))$, it follows by Lemma \ref{lem4_4} that $y \in \st_2(z)$, hence $z \in \st_2(y)$.
This shows that $\lk_2(x) \subset \st_2(y)$, that is, $x \prec_N y$.
\end{proof}

\begin{lem}\label{lem4_11}
Let $C \in S^r$.
Then $|f^r(C)| = |C|$.
\end{lem}

\begin{proof}
Let $C \in S^r$, $X=f^r(C)$, $s\in C$ and $x \in X$.
Set $U=\{t \in S \mid s \prec_M t \}$, $U_0 = \{t \in S \mid s \prec_M t \text{ and } s \not\equiv_M t \}$, $V=\{y \in T \mid x \prec_N y \}$, $V_0 = \{y \in T \mid x \prec_N y \text{ and } x \not\equiv_N y \}$.
If $E$ is a subset of $S$ (resp. of $T$) we denote by $\langle E \rangle$ the linear subspace of $L_{\K,1}[M]$ (resp. of $L_{\K,1}[N]$) spanned by $\{u_t \mid t \in E\}$.
We have $C=U \setminus U_0$, hence $\dim(\langle U \rangle) - \dim(\langle U_0 \rangle ) = |C|$.
Similarly, $\dim(\langle V \rangle) - \dim(\langle V_0 \rangle) = |X|$.
By Lemma \ref{lem4_8} and Lemma \ref{lem4_10} we have $f(\langle U\rangle)\subset \langle V\rangle$ and $f(\langle U_0\rangle)\subset\langle V_0\rangle$.
Similarly, $f^{-1}(\langle V\rangle)\subset\langle U\rangle$ and $f^{-1}(\langle V_0\rangle)\subset\langle U_0\rangle$, hence  $f(\langle U\rangle)=\langle V\rangle$ and $f(\langle U_0\rangle)=\langle V_0\rangle$.
So, $|C| = \dim(\langle U \rangle) - \dim(\langle U_0 \rangle) = \dim(\langle V \rangle) - \dim (\langle V_0 \rangle) = |X|$.
\end{proof}

\begin{lem}\label{lem4_12}
Let $C \in S^r$.
Then $\lb (f^r(C)) = \lb (C)$.
\end{lem}

\begin{proof}
Let $C \in S^r$ and let $X = f^r (C)$.
Note that, by Lemma \ref{lem4_11}, $|C| = |X|$.
If $|C| = |X| = 1$, then $\lb (C) = \lb(X) = 0$.
So, we can assume that $|C| = |X| \ge 2$.
Suppose that there exists $x \in X$ such that $s \leftrightarrow_f x$ for all $s \in C$.
Let $y \in X$, $y \neq x$.
By Lemma \ref{lem4_6} there exists $t \in S$ such that $t \leftrightarrow_f y$, and, by Corollary \ref{corl4_9}, we have  $t \in C$.
Choose $s\in C$, $s\neq t$. 
Then $s\neq t$, $x\neq y$, $s\leftrightarrow_fx$ and $t\leftrightarrow_fy$. 
Suppose that there is no $x\in X$ such that $s\leftrightarrow_fx$ for all $s\in C$. 
Then, obviously, there exist $s,t\in C$, $s\neq t$, and $x,y\in X$, $x\neq y$, such that $s\leftrightarrow_fx$ and $t\leftrightarrow_fy$.
So, there always exist $s,t \in C$, $s \neq t$, and $x,y \in X$, $x \neq y$, such that $s \leftrightarrow_f x$ and $t \leftrightarrow_f y$.
Then, by Lemma \ref{lem4_7}, $\lb (C) = m_{s,t} = n_{x,y} = \lb (X)$.
\end{proof}

\begin{proof}[Proof of Theorem \ref{thm4_1}]
Let $f : L_\K [M] \to L_\K [N]$ be an isomorphism. 
Then, by the above, $f$ induces an isomorphism $f^r : M^r \to N^r$ satisfying $|f^r(C)| = |C|$ and $\lb(f^r(C)) = \lb(C)$ for all $C \in S^r$.
We conclude by Proposition \ref{prop3_4} that $M$ and $N$ are isomorphic. 
\end{proof}


\section{Rigidity of $\AA[\EE (c,d)]$}\label{sec5}

Recall that, for $c,d\in\N$ such that $c\ge 1$, $d \ge 2$ and $\gcd (c,d)=1$, we denote by $\EE(c,d)$ the set of even Coxeter matrices $M=(m_{s,t})_{s,t\in S}$ such that $m_{s,t} \in \{2c, \infty\} \cup \{2d^r \mid r \ge 1\}$ for all $s,t \in S$, $s \neq t$.
Our aim in the present section is to prove that the family $\AA [\EE(c,d)]$ is rigid (see Theorem \ref{thm5_12} and Corollary \ref{corl5_13}).
We first prove the case $c=1$ (see Theorem \ref{thm5_1} and Corollary \ref{corl5_2}).

\begin{thm}\label{thm5_1}
Let $d \in \N$, $d \ge 2$, and $M,N \in \EE (1,d)$.
If $L[M]$ is isomorphic to $L[N]$, then $M$ is isomorphic to $N$.
\end{thm}

\begin{corl}\label{corl5_2}
Let $d \in \N$, $d \ge 2$.
Then $\AA [\EE (1,d)]$ is rigid.
\end{corl}

\begin{proof}
Let $M, N \in \EE (1,d)$.
If $A[M]$ is isomorphic to $A[N]$, then $\TGr(A[M])$ is isomorphic to $\TGr (A[N])$.
Since, by Theorem \ref{thm2_1}, $\TGr (A[M])$ is isomorphic to $L[M]$ and $\TGr(A[N])$ is isomorphic to $L[N]$, we conclude by Theorem \ref{thm5_1} that $M$ is isomorphic to $N$.
\end{proof}

We turn now to prove Theorem \ref{thm5_1}. 
So, we take two Coxeter matrices $M = (m_{s,t})_{s,t \in S}$ and $N=(n_{x,y})_{x,y \in T}$ lying in $\EE(1,d)$ and an isomorphism $f : L[M] \to L[N]$.
Our aim is to prove that $M$ and $N$ are isomorphic.

We fix a prime number $p$ which divides $d$.
The \emph{$p$-support} of an element $\alpha = \sum_{s \in S} a_s u_s \in L_1[M]$ is defined to be $\supp_p(\alpha) = \{ s \in S \mid \gcd (p,a_s)=1  \}$.
For $s \in S$ and $x \in T$ we set $s \leftrightarrow_{p,f} x$ if $s \in \supp_p(f^{-1} (u_x))$ and $x \in \supp_p (f(u_s))$.

\begin{lem}\label{lem5_3}
Let $s \in S$. 
There exists $x \in T$ such that $s \leftrightarrow_{p,f} x$.
\end{lem}

\begin{proof}
We argue like in the proof of Lemma \ref{lem4_6}.
For $s \in S$ we set $f(u_s) = \sum_{x \in T} a_{s,x} u_x$, and for $x \in T$ we set $f^{-1}(u_x) = \sum_{s \in S} b_{x,s}u_s$.
Let $s \in S$.
Then 
\[
u_s = (f^{-1} \circ f)(u_s) = \sum_{x \in T} \sum_{t \in S} a_{s,x}b_{x,t} u_t\,,
\]
hence there exists $x \in T$ such that $\gcd (p,a_{s,x})=1$ and $\gcd (p,b_{x,s})=1$.
Then $s \in \supp_p(f^{-1}(u_x))$ and $x \in \supp_p (f (u_s))$, that is, $s \leftrightarrow_{p,f} x$.
\end{proof}

We set $L^{(d)}_1[M] = L_1 [M]$, $L^{(d)}_2 [M] = d\,L_2[M]$, and $L^{(d)}[M] = L^{(d)}_1[M] \oplus L^{(d)}_2 [M]$, and we define the Lie bracket $[\cdot ,\cdot ]^{(d)}$ of $L^{(d)}[M]$ as follows. 
For $\alpha, \beta \in L^{(d)}_1 [M]$ we set $[\alpha,\beta]^{(d)} = d\,[\alpha,\beta]$, where $[\cdot,\cdot]$ denotes the Lie bracket of $L[M]$.
On the other hand, we set $[\alpha, \beta]^{(d)}=0$ for all $(\alpha, \beta) \in L_2^{(d)}[M] \times L^{(d)}[M]$.
We define the Lie algebra $L^{(d)}[N]$ in the same way. 
It is easily seen that $f$ induces an isomorphism $f^{(d)} : L^{(d)}[M] \to L^{(d)}[N]$.

Let $M^{(d)}=(m_{s,t}^{(d)})_{s,t \in S}$ be the Coxeter matrix defined as follows. 
\begin{itemize}
\item
$m_{s,t}^{(d)}= 2$ if $m_{s,t}\in \{ 2,2d\}$.
\item
Let $r \in \N$, $r \ge 2$.
Then $m_{s,t}^{(d)}= 2d^{r-1}$ if $m_{s,t}=2d^r$.
\item
$m_{s,t}^{(d)}= \infty$ if $m_{s,t}=\infty$.
\end{itemize}
We define $N^{(d)}$ in the same way. 
Note that $M^{(d)}$ and $N^{(d)}$ are even Coxeter matrices lying in $\EE(1,d)$.
The proof of the following lemma is left to the reader.

\begin{lem}\label{lem5_4}
\begin{itemize}
\item[(1)]
We have $L^{(d)} [M] \simeq L [M^{(d)}]$ and $L^{(d)} [N] \simeq L[N^{(d)}]$.
\item[(2)]
Let $s \in S$ and $x \in T$.
We have $s \leftrightarrow_{p,f} x$ if and only if $s \leftrightarrow_{p,f^{(d)}} x$.
\end{itemize}
\end{lem}

We define a right-angled Coxeter matrix $M_{(p)}= (m_{(p),s,t})_{s,t \in S}$ as follows. 
\begin{itemize}
\item
$m_{(p),s,t}= 2$ if $m_{s,t}=2$.
\item
$m_{(p),s,t} = \infty$ if $m_{s,t} \in \{2d^r \mid r \in \N\,,\ r \ge 1\} \cup \{\infty\}$.
\end{itemize}
We define the right-angled Coxeter matrix $N_{(p)}$ in the same way. 
Let $\F_p = \Z/p\Z$, and let $f_{(p)} : \F_p \otimes L[M] \to \F_p \otimes L[N]$ be the isomorphism induced by $f$.
Again, the proof of the following is left to the reader.

\begin{lem}\label{lem5_5}
\begin{itemize}
\item[(1)]
We have $\F_p \otimes L [M] \simeq L_{\F_p} [M_{(p)}]$ and $\F_p \otimes L[N] \simeq L_{\F_p} [N_{(p)}]$.
\item[(2)]
Let $s \in S$ and $x \in T$.
We have $s \leftrightarrow_{p,f} x$ if and only if $s \leftrightarrow_{f_{(p)}} x$.
\end{itemize}
\end{lem}

\begin{lem}\label{lem5_6}
Let $s,t \in S$, $s \neq t$, and $x,y \in T$, $x \neq y$, such that $s \leftrightarrow_{p,f} x$ and $t \leftrightarrow_{p,f} y$.
Then $m_{s,t} = n_{x,y}$.
\end{lem}

\begin{proof}
Let $\rho \in \N$ such that $E_{2d^\rho} (M) \neq \emptyset$ or $E_{2d^\rho} (N) \neq \emptyset$, and $E_{2d^k} (M) = E_{2d^k}(N) = \emptyset$ for all $k>\rho$ (we set $\rho=0$ if $m_{s,t}=\infty$ for all $s,t \in S$, $s \neq t$, and $n_{x,y} = \infty$ for all $x,y \in T$, $x \neq y$).
We argue by induction on $\rho$.
If $\rho=0$, then $M$ and $N$ are both right-angled Coxeter matrices and the result follows from Lemma \ref{lem4_7}.
So, we can assume that $\rho \ge 1$ and that the inductive hypothesis holds.
Suppose that $m_{s,t}=2$.
Then $m_{(p),s,t}=2$ and, by Lemma \ref{lem5_5}, $s \leftrightarrow_{f_{(p)}} x$ and $t \leftrightarrow_{f_{(p)}} y$, hence, by Lemma \ref{lem4_7}, $n_{(p),x,y} = 2$, and therefore $n_{x,y}=2$.
We show in the same way that, if $n_{x,y}=2$, then $m_{s,t}=2$.
So, we can assume that $m_{s,t} \ge 2d$ and $n_{x,y}\ge 2d$.
By Lemma \ref{lem5_4}, $s \leftrightarrow_{p,f^{(d)}} x$ and $t \leftrightarrow_{p,f^{(d)}} y$, hence, by the inductive hypothesis, $m_{s,t}^{(d)}=n_{x,y}^{(d)}$.
If $m_{s,t}^{(d)}=n_{x,y}^{(d)} = \infty$, then $m_{s,t}=n_{x,y}=\infty$.
If $m_{s,t}^{(d)}=n_{x,y}^{(d)} < \infty$, then $m_{s,t}=n_{x,y}=d\,m_{s,t}^{(d)}=d\,n_{x,y}^{(d)}$, since $m_{s,t} \ge 2d$ and $n_{x,y}\ge 2d$.
\end{proof}

\begin{lem}\label{lem5_7}
Let $s,t \in S$ and $x,y \in T$ such that $s \leftrightarrow_{p,f} x$ and $t \leftrightarrow_{p,f} y$.
We have $s \prec_M t$ if and only if $x \prec_N y$.
\end{lem}

\begin{proof}
Let $\rho \in \N$ such that $E_{2d^\rho}(M) \neq \emptyset$ or $E_{2 d^\rho} (N) \neq \emptyset$, and $E_{2d^k}(M) = E_{2d^k} (N) = \emptyset$ for all $k > \rho$ (again, we set $\rho=0$ if $m_{s,t}=\infty$ for all $s,t \in S$, $s \neq t$, and $n_{x,y}=\infty$ for all $x,y\in T$, $x\neq y$).
We argue by induction on $\rho$.
Assume that $\rho=0$, that is, $M$ and $N$ are both right-angled Coxeter matrices. 
Then $f$ induces an isomorphism $f_{\F_p}:L_{\F_p}[M]=\F_p\otimes L[M]\to\F_p\otimes L[N]=L_{\F_p}[N]$, hence, by Lemma \ref{lem4_8}, we have $s \prec_M t$ if and only if $x \prec_N y$.

Now, assume that $\rho \ge 1$ and that the inductive hypothesis holds. 
Suppose that $s \prec_M t$.
We take $m \in \N$, $m \ge 2$, and we prove that $\lk_m(x) \subset \st_m(y)$.
This will prove that $x \prec_N y$.
Since $M,N \in \EE(1,d)$, we can assume that $m \in \{ 2d^k \mid 0 \le k \le \rho \}$.

Suppose that $m=2$.
For $r \in S$ we denote by $\lk_{(p),2} (r)$ the $2$-link of $r$ and by $\st_{(p),2} (r)$ the $2$-star of $r$ with respect to $M_{(p)}$.
Similarly, for $z \in T$ we denote by $\lk_{(p),2} (z)$ the $2$-link of $z$ and by $\st_{(p),2} (z)$ the $2$-star of $z$ with respect to $N_{(p)}$.
Observe that $\lk_2(s) = \lk_{(p),2} (s)$ and $\st_2(t) = \st_{(p),2} (t)$, hence $\lk_{(p),2} (s) \subset \st_{(p),2} (t)$, that is, $s \prec_{M_{(p)}} t$.
Moreover, by Lemma \ref{lem5_5}, $s \leftrightarrow_{f_{(p)}} x$ and $t \leftrightarrow_{f_{(p)}} y$, hence, by Lemma \ref{lem4_8}, $x \prec_{N_{(p)}} y$.
This means that $\lk_2 (x) = \lk_{(p),2} (x) \subset \st_{(p),2} (y) = \st_2 (y)$.

Assume that $m \ge 2d$, that is, $m=2d^{k+1}$ for some $k\in\{0,\dots,\rho-1\}$.
For $r \in S$ and $0 \le k \le \rho-1$ we denote by $\lk^{(d)}_{2d^k} (r)$ the $2d^k$-link of $r$ and by $\st^{(d)}_{2d^k} (r)$ the $2d^k$-star of $r$ with respect to $M^{(d)}$.
Similarly, for $z \in T$ and $0 \le k \le \rho-1$ we denote by $\lk^{(d)}_{2d^k} (z)$ the $2d^k$-link of $z$ and by $\st^{(d)}_{2d^k} (z)$ the $2d^k$-star of $z$ with respect to $N^{(d)}$.
Observe that $\lk_{2d^{k+1}}(s) = \lk^{(d)}_{2d^k} (s)$ and $\st_{2d^{k+1}} (t) = \st^{(d)}_{2d^k} (t)$ for all $0 \le k \le \rho-1$, hence $\lk^{(d)}_{2d^k} (s) \subset \st^{(d)}_{2d^k} (t)$ for all $0 \le k \le \rho-1$, and therefore $s \prec_{M^{(d)}} t$.
Moreover, by Lemma \ref{lem5_4}, $s \leftrightarrow_{p,f^{(d)}} x$ and $t \leftrightarrow_{p,f^{(d)}} y$, hence, by the inductive hypothesis, $x \prec_{N^{(d)}} y$.
This implies that $\lk_{2d^{k+1}} (x) = \lk^{(d)}_{2d^k} (x) \subset \st^{(d)}_{2d^k} (y) = \st_{2d^{k+1}} (y)$ for all $0 \le k \le \rho-1$.
So, $\lk_m(x) \subset \st_m(y)$.
\end{proof}

\begin{corl}\label{corl5_8}
Let $s,t \in S$ and $x,y \in T$ such that $s \leftrightarrow_{p,f} x$ and $t \leftrightarrow_{p,f} y$.
We have $s \equiv_M t$ if and only if $x \equiv_N y$.
\end{corl}

Thanks to Corollary \ref{corl5_8} we can define a map $f^r : S^r \to T^r$ as follows.
Let $C \in S^r$.
We choose $s \in C$ and $x \in T$ such that $s \leftrightarrow_{p,f} x$ and we set $f^r(C) = X$, where $X$ is the equivalence class of $x$ for the relation $\equiv_N$.
It is clear that $f^r$ is invertible and $(f^r)^{-1} = (f^{-1})^r$.
Moreover, by Lemma \ref{lem5_6}, $m_{C,D}^r = n_{f^r(C),f^r(D)}^r$ for all $C,D \in S^r$, hence $f^r$ determines an isomorphism from $M^r$ to $N^r$.

\begin{lem}\label{lem5_9}
Let $s \in S$ and $x,y \in T$ such that $s \leftrightarrow_{p,f} x$ and $y \in \supp_p (f(u_s))$.
Then $x \prec_N y$.
\end{lem}

\begin{proof}
Let $\rho \in \N$ such that $E_{2d^\rho}(M) \neq \emptyset$ or $E_{2 d^\rho} (N) \neq \emptyset$, and $E_{2d^k}(M) = E_{2d^k} (N) = \emptyset$ for all $k > \rho$ (and, as ever, $\rho=0$ if $m_{s,t} = \infty$ for all $s,t \in S$, $s \neq t$, and $n_{x,y}=\infty$ for all $x,y\in T$, $x\neq y$).
We argue by induction on $\rho$.
Suppose that $\rho=0$, that is, $M$ and $N$ are both right-angled Coxeter matrices. 
Then $f$ induces an isomorphism $f_{(p)} : L_{\F_p} [M] = \F_p \otimes L[M] \to \F_p \otimes L[N] = L_{\F_p} [N]$ hence, by Lemma \ref{lem4_10}, $x \prec_N y$.

Now, we assume that $\rho \ge 1$ and that the inductive hypothesis holds. 
We take $m \in \N$, $m \ge 2$, and we show that $\lk_m(x) \subset \st_m(y)$.
This will show that $x \prec_N y$.
Again, since $M$ and $N$ lie in $\EE(1,d)$, we can suppose that $m \in \{ 2d^k \mid 0 \le k \le \rho \}$.

Suppose first that $m=2$.
For $z \in T$ we denote by $\lk_{(p),2} (z)$ the $2$-link of $z$ and by $\st_{(p),2} (z)$ the $2$-star of $z$ with respect to $N_{(p)}$.
By Lemma \ref{lem5_5}, $s \leftrightarrow_{f_{(p)}} x$ hence, by Lemma \ref{lem4_10}, $x \prec_{N_{(p)}} y$.
This implies that $\lk_2 (x) = \lk_{(p),2} (x) \subset \st_{(p),2} (y) = \st_2 (y)$.

Now, suppose that $m \ge 2d$, that is, $m=2d^{k+1}$ for some $k\in\{0,\dots,\rho-1\}$.
For $z \in T$ and $0 \le k \le \rho-1$ we denote by $\lk^{(d)}_{2d^k} (z)$ the $2d^k$-link of $z$ and by $\st^{(d)}_{2d^k} (z)$ the $2d^k$-star of $z$ with respect to $N^{(d)}$.
By Lemma \ref{lem5_4}, $s \leftrightarrow_{p,f^{(d)}} x$ hence, by the inductive hypothesis, $x \prec_{N^{(d)}} y$.
This implies that $\lk_{2d^{k+1}} (x) = \lk^{(d)}_{2d^k} (x) \subset \st^{(d)}_{2d^k} (y) = \st_{2d^{k+1}} (y)$ for all $0 \le k \le \rho-1$.
So, $\lk_m(x) \subset \st_m(y)$.
\end{proof}

\begin{lem}\label{lem5_10}
Let $C \in S^r$.
Then $|f^r(C)| = |C|$.
\end{lem}

\begin{proof}
The proof is almost identical to the one of Lemma \ref{lem4_11}.
Let $C \in S^r$, $X=f^r(C)$, $s\in C$ and $x \in X$.
Set $U=\{t \in S \mid s \prec_M t \}$, $U_0 = \{t \in S \mid s \prec_M t \text{ and } s \not\equiv_M t \}$, $V=\{y \in T \mid x \prec_N y \}$, $V_0 = \{y \in T \mid x \prec_N y \text{ and } x \not\equiv_N y \}$.
If $E$ is a subset of $S$ (resp. of $T$) we denote by $\langle E \rangle$ the $\Z$-submodule of $L_1[M]$ (resp. of $L_1 [N]$) spanned by $\{u_t \mid t \in E\}$.
We have $C=U \setminus U_0$, hence $\rk(\langle U \rangle) - \rk(\langle U_0 \rangle ) = |C|$.
Similarly, $\rk(\langle V \rangle) - \rk(\langle V_0 \rangle) = |X|$.
On the other hand, by Lemma \ref{lem5_7} and Lemma \ref{lem5_9}, $f(\langle U \rangle) \subset \langle V \rangle$ and $f(\langle U_0 \rangle) \subset \langle V_0 \rangle$.
Similarly, $f^{-1}(\langle V\rangle)\subset\langle U\rangle$ and $f^{-1}(\langle V_0\rangle)\subset\langle U_0\rangle$, hence $f(\langle U\rangle)=\langle V\rangle$ and $f(\langle U_0\rangle)=\langle V_0\rangle$.
So, $|C| = \rk(\langle U \rangle) - \rk(\langle U_0 \rangle) = \rk(\langle V \rangle) - \rk (\langle V_0 \rangle) = |X|$.
\end{proof}

\begin{lem}\label{lem5_11}
Let $C \in S^r$.
Then $\lb (f^r(C)) = \lb (C)$.
\end{lem}

\begin{proof}
We argue in the same way as in the proof of Lemma \ref{lem4_12}.
Let $C \in S^r$, and let $X = f^r (C)$.
Note that, by Lemma \ref{lem5_10}, $|C| = |X|$.
If $|C| = |X| = 1$, then $\lb (C) = \lb(X) = 0$.
So, we can assume that $|C| = |X| \ge 2$.
Suppose there exists $x \in X$ such that $s \leftrightarrow_{p,f} x$ for all $s \in C$.
Let $y \in X$, $y \neq x$.
By Lemma \ref{lem5_3} there exists $t \in S$ such that $t \leftrightarrow_{p,f} y$ and, by Corollary \ref{corl5_8}, $t \in C$.
Take $s \in C$ such that $s \neq t$.
Then $s \neq t$, $x \neq y$, $s \leftrightarrow_{p,f} x$ and $t \leftrightarrow_{p,f} y$.
If there is no $x \in X$ such that $s \leftrightarrow_{p,f} x$ for all $s \in C$, then, obviously, there exist $s,t \in C$, $s \neq t$, and $x,y \in X$, $x \neq y$, such that $s \leftrightarrow_{p,f} x$ and $t \leftrightarrow_{p,f} y$. 
So, there always exist $s,t \in C$, $s \neq t$, and $x,y \in X$, $x \neq y$, such that $s \leftrightarrow_{p,f} x$ and $t \leftrightarrow_{p,f} y$.
Then, by Lemma \ref{lem5_6}, $\lb (C) = m_{s,t} = n_{x,y} = \lb (X)$.
\end{proof}

\begin{proof}[Proof of Theorem \ref{thm5_1}]
Let $f : L [M] \to L [N]$ be an isomorphism.
Then, by the above, there exists an isomorphism $f^r : M^r \to N^r$ such that $|f^r(C)| = |C|$ and $\lb(f^r(C)) = \lb(C)$ for all $C \in S^r$.
We conclude by Proposition \ref{prop3_4} that $M$ and $N$ are isomorphic. 
\end{proof}

We turn now to the general case. 

\begin{thm}\label{thm5_12}
Let $c,d \in \N$ such that $c \ge 1$, $d \ge 2$ and $\gcd (c,d)=1$, and let $M,N \in \EE (c,d)$.
If $L[M]$ is isomorphic to $L[N]$, then $M$ is isomorphic to $N$.
\end{thm}

The following corollary can be proved from Theorem \ref{thm5_12} in the same way as Corollary \ref{corl5_2} is proved from Theorem \ref{thm5_1}.

\begin{corl}\label{corl5_13}
Let $c,d \in \N$ such that $c \ge 1$, $d \ge 2$ and $\gcd (c,d)=1$.
Then $\AA [\EE (c,d)]$ is rigid.
\end{corl}

\begin{proof}[Proof of Theorem \ref{thm5_12}]
Let $f : L[M] \to L[N]$ be an isomorphism.
We define a graded Lie $\Z$-algebra $L^{(c)} [M]$ as follows. 
We set $L^{(c)}_1[M] = L_1 [M]$, $L^{(c)}_2[M] = c\,L_2[M]$, and $L^{(c)} [M] = L^{(c)}_1[M] \oplus L^{(c)}_2 [M]$.
We define the Lie bracket $[\cdot ,\cdot ]^{(c)}$ of $L^{(c)} [M]$ as follows. 
For $\alpha, \beta \in L_1^{(c)}[M]$ we set $[\alpha,\beta]^{(c)} = c\,[\alpha,\beta]$, where $[\cdot,\cdot]$ denotes the Lie bracket of $L [M]$.
On the other hand, we set $[\alpha, \beta]^{(c)}=0$ for all $(\alpha, \beta) \in L_2^{(c)}[M] \times L^{(c)}[M]$.
We define $L^{(c)}[N]$ in the same way.
It is easily seen that $f$ induces an isomorphism $f^{(c)} : L^{(c)}[M] \to L^{(c)}[N]$.

Let $M^{(c)}=(m_{s,t}^{(c)})_{s,t \in S}$ be the Coxeter matrix defined as follows.
\begin{itemize}
\item
$m_{s,t}^{(c)}= 2$ if $m_{s,t}=2c$.
\item
$m_{s,t}^{(c)}= m_{s,t}$ if $m_{s,t} \neq 2c$. 
\end{itemize}
We define $N^{(c)}$ in the same way. 
It is clear that $M^{(c)}$ and $N^{(c)}$ are both even Coxeter matrices lying in $\EE(1,d)$, that $M$ is isomorphic to $N$ if and only if $M^{(c)}$ is isomorphic to $N^{(c)}$, that $L^{(c)}[M] \simeq L[M^{(c)}]$, and that $L^{(c)}[N] \simeq L [N^{(c)}]$.
By Theorem \ref{thm5_1}, $M^{(c)}$ and $N^{(c)}$ are isomorphic, hence $M$ and $N$ are isomorphic. 
\end{proof}


\section{An example}\label{sec6}

Consider the following even Coxeter matrices.
\[
M_0=\left( \begin{array}{ccc}
1 & 6 & 2 \\
6 & 1 & 10 \\
2 & 10 & 1
\end{array} \right)\,,\
N_0 = \left( \begin{array}{ccc}
1 & 2 & 2 \\
2 & 1 & 30\\
2 & 30 & 1
\end{array} \right)\,.
\]
It is obvious that $M_0$ and $N_0$ are not isomorphic. 
Furthermore, we also have the following.

\begin{lem}\label{lem6_1}
The Artin groups $A[M_0]$ and $A[N_0]$ are not isomorphic.
\end{lem}

\begin{proof}
The group $A[M_0]$ is a two-dimensional Artin group in the sense of Charney--Davis \cite{ChaDav1}, hence, by Charney--Davis \cite{ChaDav1} \cite[Corollary 1.4.2]{ChaDav2}, the cohomological dimension of $A[M_0]$ is $2$.
On the other hand, $A[N_0]$ is an Artin group of spherical type of rank $3$, hence, by Charney--Davis \cite[Corollary 1.4.2]{ChaDav2}, the cohomological dimension of $A[N_0]$ is $3$. 
So, $A[M_0]$ and $A [N_0]$ are not isomorphic. 
\end{proof}

However, we have the following.

\begin{prop}\label{lem6_2}
The Lie algebras $L[M_0]$ and $L[N_0]$ are isomorphic.
\end{prop}

\begin{proof}
We denote by $\UU = \{u_1, u_2, u_3\}$ the standard $\Z$-basis of $L_1[M_0]$ and by $\UU'=\{u_1',u_2',u_3'\}$ the standard $\Z$-basis of $L_1[N_0]$.
On the other hand, we denote by $\VV=\{v_{1,2},v_{2,3}\}$ the standard generating set of $L_2[M_0]$ and by $\VV'=\{v_{2,3}'\}$ the standard generating set of $L_2[N_0]$. 
Let $f_1:L_1[M_0]\to L_1[N_0]$ and $g_1:L_1[N_0]\to L_1[M_0]$ be the $\Z$-linear maps whose matrices relative to the basis $\UU$ and $\UU'$ are $F$ and $G$, respectively, where 
\[
F = \left( \begin{array}{ccc}
2 & 0 & 1 \\
0 & 1 & 0 \\
5 & 0 & 3
\end{array} \right)\,,\
G=\left( \begin{array}{ccc}
3 & 0 & -1 \\
0 & 1 & 0 \\
-5 & 0 & 2
\end{array} \right)\,.
\]
A direct calculation shows that $f_1\circ g_1=\id$ and $g_1\circ f_1=\id$, hence $f_1$ is an isomorphism of $\Z$-modules and $g_1$ is the inverse of $f_1$.
We define $f_2:L_2[M_0]\to L_2[N_0]$ by $f_2(v_{1,2})=-5v_{2,3}'$ and $f_2(v_{2,3}) = 3v_{2,3}'$.
It is easily checked that $f_2$ is well-defined, it is invertible, and its inverse $g_2:L_2[N_0]\to L_2[M_0]$ is defined by $g_2(v_{2,3}')=v_{1,2}+2v_{2,3}$.
We check with a direct calculation that $[f_1(u_i),f_1(u_j)]=f_2([u_i,u_j])$ for all $i,j\in\{1,2,3\}$, hence $f_1$ and $f_2$ induce an isomorphism $f:L[M_0]\to L[N_0]$ of Lie $\Z$-algebras. 
\end{proof}



\end{document}